%% file: KR2.tex
\documentclass[11pt]{amsart}
\usepackage{amssymb,latexsym,amsmath,longtable,mathdots,comment,cite}
\usepackage[mathscr]{eucal}
\usepackage{bbm}
\usepackage{color,lscape}
\usepackage[all]{xy}

\numberwithin{equation}{section}
\numberwithin{table}{section}
\numberwithin{figure}{section}

\input{macros.tex} 
\input{thms.tex} 

\theoremstyle{theorem}
\newtheorem*{mt}{Theorem \ref{T:smooth}}

\begin{document}

\title{Classification of smooth horizontal Schubert varieties}

\author[Kerr]{Matt Kerr}
\email{matkerr@math.wustl.edu}
\address{Department of Mathematics, Washington University in St. Louis, Campus Box 1146, One Brookings Drive, St. Louis, MO 63130-4899}
\thanks{Kerr is partially supported by NSF grants DMS-1068974, 1259024, and 1361147.}

\author[Robles]{Colleen Robles}
\email{robles@math.duke.edu}
\address{Mathematics Department, Duke University, Box 90320, Durham, NC  27708-0320} 
\thanks{Robles is partially supported by NSF grants DMS 02468621 and 1361120.  This work was undertaken while Robles was a member of the Institute for Advanced Study; she thanks the institute for a wonderful working environment and the Robert and Luisa Fernholz Foundation for financial support.}

\date{\today}

\begin{abstract}
We show that the smooth horizontal Schubert subvarieties of a rational homogeneous variety $G/P$ are homogeneously embedded cominuscule $G'/P'$, and are classified by subdiagrams of a Dynkin diagram.  This generalizes the classification of smooth Schubert varieties in cominuscule $G/P$. 
\end{abstract}


\maketitle



\section{Introduction}

The main result of this paper is a characterization of the smooth Schubert varieties of a rational homogeneous variety $G/P$ that are integrals of a canonically defined subbundle of $T(G/P)$.  Here $G$ is a complex, semisimple Lie group, $P$ is a parabolic subgroup of $G$, and $T(G/P)$ is the holomorphic tangent bundle.  Every rational homogeneous variety $G/P$ admits a unique minimal, bracket--generating, $G$--homogeneous subbundle $T^1 \subset T(G/P)$ of the (holomorphic) tangent bundle.  The variety $G/P$ is cominuscule if and only if $T^1 = T(G/P)$; equivalently, $G/P$ admits the structure of a Hermitian symmetric space.  These are the simplest rational homogeneous varieties; examples include the Grassmannian $\tGr(k,\bC^n)$ of $k$--planes in $\bC^n$.  From the work of Lakshmibai--Weyman, Brion--Polo, and J.~Hong we have 

\begin{theorem}[{\cite{MR1703350, MR2276624, MR1080976}}] \label{T:comin}
In the case that $G/P$ is cominuscule, the smooth Schubert varieties $X \subset G/P$ are homogeneously embedded cominuscule $G'/P'$, and are indexed by subdiagrams of the Dynkin diagram $\sD$ of $G$.
\end{theorem}  

\noindent Hong--Mok generalized Theorem \ref{T:comin}, using the deformation theory to prove

\begin{theorem}[{\cite{HongMok2013}}] \label{T:HM}
Suppose that $G$ is simple and $P$ is a maximal parabolic with the property that the associated simple root is not short.  Then the smooth Schubert varieties $X \subset G/P$ are homogeneously embedded $G'/P'$, and are indexed by subdiagrams of the Dynkin diagram $\sD$ of $G$.
\end{theorem}

The main result (Theorem \ref{T:smooth}) of the paper is another generalization of Theorem \ref{T:comin}.  We say that a Schubert variety $X \subset G/P$ is ``horizontal'' if it is tangent to $T^1$ at smooth points (\S\ref{S:HSV}).

\begin{theorem} \label{T:smooth}
Let $X \subset G/P$ be a horizontal Schubert variety.  If $X$ is smooth, then $X$ is a product of homogeneously embedded, rational homogeneous subvarieties $X(\sD') \subset G/P$ corresponding to subdiagrams $\sD' \subset \sD$.  Moreover, each $X(\sD')$ is cominuscule.
\end{theorem}

\noindent The proof of Theorem \ref{T:smooth} is based on the characterization of horizontal Schubert varieties (HSV) in \cite{MR3217458}, and Theorem \ref{T:HM}.  Note that, in general, a smooth Schubert variety $X \subset {G/P}$ need not be homogeneous (Example \ref{eg:SG}).  As will be seen in the course of the proof, the subdiagrams $\sD'$ may be explicitly characterized.

As an application of the main result we characterize the maximal parabolics $P \subset G$ with the property that the Schubert variety swept out by the horizontal lines passing through a point is itself horizontal.  This, and the material developed in \S\ref{S:lines}, is motivated by anticipated applications to a current project of the authors \cite{KR1}.

\section{Horizontal Schubert varieties}

\subsection{Flag manifolds} \label{S:rhv}

Let $G$ be a connected, complex semisimple Lie group, and let $P \subset G$ be a parabolic subgroup.  The homogeneous manifold $G/P$ admits the structure of a rational homogeneous variety as follows.  Fix a choice of \emph{Cartan} and \emph{Borel subgroups} $H \subset B \subset P$.  Let $\fh \subset \fb \subset\fp\subset\fg$ be the associated Lie algebras.  The choice of Cartan determines a set of \emph{roots} $\Delta = \Delta(\fg,\fh) \subset \fh^*$.  Given a root $\a\in \Delta$, let $\fg^\a \subset \fg$ denote the \emph{root space}.  Given a subspace $\fs \subset \fg$, let 
$$
  \Delta(\fs) \ \dfn \ \{ \a\in\Delta \ | \ \fg^\a \subset \fs \} \,.
$$ 
The choice of Borel determines \emph{positive roots} $\Delta^+ = \Delta(\fb) = \{ \a\in\Delta \ | \ \fg^\a \subset \fb \}$.  Let $\sS = \{\a_1,\ldots,\a_r\}$ denote the \emph{simple roots}, and set 
\begin{equation} \label{E:I}
  I \ = \ I(\fp) \ \dfn \ \{ i \ | \ \fg^{-\a_i} \not\subset \fp \}\,.
\end{equation}
Note that 
$$
  I(\fb) \ = \ \{1,\ldots,r\} \,,
$$
and $I = \{\tti\}$ consists of single element if and only if $\fp$ is a \emph{maximal parabolic}.

Let $\{ \w_1,\ldots,\w_r\}$ denote the \emph{fundamental weights} of $\fg$ and let $V$ be the irreducible $\fg$--representation of highest weight 
\begin{equation} \label{E:mu}
  \m \ = \ \m_I \ \dfn \ \sum_{i \in I} \w_i \,.
\end{equation}
Assume that the representation $\fg \to \tEnd(V)$ `integrates' to a representation $G \to \tAut(V)$ of $G$.  (This is always the case if $G$ is simply connected.)  Let $o \in \bP V$ be the \emph{highest weight line} in $V$.  The $G$--orbit $G \cdot o \subset \bP V$ is the \emph{minimal homogeneous embedding} of $G/P$ as a rational homogeneous variety.

\begin{remark}[Non-minimal embeddings] \label{R:nonminemb}
More generally, suppose that $V$ is the irreducible $G$--representation of highest weight $\tilde \m = \sum_{i \in I} a^i\,\w_i$ with $0 < a^i \in \bZ$.  Again, the $G$--orbit of the highest weight line $o \in \bP V$ is a homogeneous embedding of $G/P$.  We write $G/P \inj \bP V$.  The embedding is minimal if and only if $a^i=1$ for all $i \in I$.  For example, the Veronese re-embedding $v_d(\bP^n) \subset \bP\,\tSym^d\bC^{n+1}$ of $\bP^n$ is if minimal if and only if $d = 1$.  (Here $V = \tSym^d\bC^{n+1}$ has highest weight $d\w_1$.)
\end{remark}

The variety $G/P$ is sometimes indicated by circling the nodes of the Dynkin diagram (Appendix \ref{S:dynkin}) corresponding to the index set $I(\fp)$.

\subsection{Schubert varieties} \label{S:schub}

In the case that $G$ is classical (one of $\tSL_n\bC$, $\tSO_n\bC$ or $\tSp_{2n}\bC$), the Schubert varieties of a flag variety $G/P$ may be described geometrically as degeneracy loci, \cf Example \ref{eg:schub1}.  However, these descriptions, aside from not generalizing easily to the exceptional groups, are type dependent.  We will utilize a representation theoretic description of the Schubert varieties that allows a uniform treatment across all flag varieties.  This section does little more than establish notation for our discussion of Schubert varieties.  The reader interested in greater detail is encouraged to consult \cite{MR3217458} and the references therein.

Given a simple root $\a_i\in\sS$, let $(i) \in \tAut(\fh^*)$ denote the corresponding \emph{simple reflection}.  The \emph{Weyl group} $\sW \subset \tAut(\fh^*)$ of $\fg$ is the group generated by the simple reflections $\{ (i) \ | \ \a_i\in\sS \}$.  A composition of simple reflections $(i_1) \circ (i_2) \circ\cdots\circ(i_t)$, which are understood to act on the left, is written $(i_1 i_2 \cdots i_t) \in \sW$.  The \emph{length} of a Weyl group element $w$ is the minimal number 
$$
  |w| \ \dfn \ \tmin\{ \ell \ | \ w = (i_1 i_2 \cdots i_\ell) \}
$$
of simple reflections necessary to represent $w$.

Let $\sW_\fp \subset \sW$ be the subgroup generated by the simple reflections $\{(i) \ | \ i \not \in I\}$.  Then $\sW_\fp$ is naturally identified with the Weyl group of $\fg^0_\mathrm{ss}$.  The rational homogeneous variety $G/P$ decomposes into a finite number of $B$--orbits 
$$
  G/P \ = \ \bigcup_{\sW_\fp w \in \sW_\fp\backslash \sW} B w^{-1} o 
$$
which are indexed by the right cosets $\sW_\fp\backslash\sW$.  The \emph{$B$--Schubert varieties} of $G/P$ are the Zariski closures
$$
  X_w \ \dfn \ \overline{B w^{-1} o }\,.
$$

\begin{remark}\label{R:schubStab}
Observe that the stabilizer $P_w$ of $X_w$ in $G$ contains $B$, and is therefore a parabolic subgroup of $G$.
\end{remark}

\noindent 
More generally, any $G$--translate $g X_w$ is a Schubert variety (of type $\sW_\fp w$).  Define a partial order on $\sW_\fp\backslash \sW$ by defining $\sW_\fp w \le \sW_\fp v$ if $X_w \subset X_v$; then $\sW_\fp\backslash\sW$ is the \emph{Hasse poset}.

Each right coset $\sW_\fp\backslash \sW$ admits a unique representative of minimal length; let 
$$
  \sW^\fp_\tmin \ \simeq \ \sW_\fp \backslash W
$$ 
be the set of minimal length representatives.  Given $w \in \sW^\fp_\tmin$, the Schubert variety $w X_w$ is the Zariski closure of $N_w \cdot o$, where $N_w \subset G$ is a unipotent subgroup with nilpotent Lie algebra 
\begin{equation} \label{E:nw}
  \fn_w \ \dfn \ \bigoplus_{\a\in \Delta(w)} \fg^{-\a} \ \subset \ \fg^{-}
\end{equation}
given by 
\begin{equation} \label{E:Dw}
  \Delta(w) \ \dfn \ \Delta^+ \,\cap \, w(\Delta^-) \,.
\end{equation}
Moreover, $N_w \cdot o$ is an affine cell isomorphic to $\fn_w$, and $\tdim\,X_w = \tdim\,\fn_w = |\Delta(w)|$.  Indeed 
$$
  T_o (w X_w) \ = \ \fn_w \,.
$$
For any $w \in \sW^\fp_\tmin$ we have 
\begin{equation}\label{E:len=dim}
  |w| \ = \ |\Delta(w)| \ = \ \tdim\,X_w \,.
\end{equation}
(The first equality holds for all $w \in \sW$, \cf\cite[Proposition 3.2.14(3)]{MR2532439}.)

\begin{example}[Schubert varieties in $\tOG(2,\bC^m)$] \label{eg:schub1}
Let $\nu$ be a nondegenerate, symmetric bilinear form on $\bC^m$, and let 
\[
  \tOG(2,\bC^m) \ \dfn \ \{ E \in \tGr(2,\bC^m) \ | \ \left.\n\right|_{E} = 0\}
\]
be the orthogonal grassmannian of $\n$--isotropic $2$--planes in $\bC^m$.  Fix a basis $\{ e_1 , \ldots , e_m\}$ of $\bC^m$ with the property that
\[
  \n( e_a , e_b ) \ = \d^{m+1}_{a+b} \,.
\]
Then 
\[
  \sF^p \ \dfn \ \tspan_\bC\{ e_1 , \ldots , e_p \} \ \subset \ \bC^m
\]
defines a flag $\sF^\sb$ with the property that 
\begin{equation} \label{E:nu-iso-F}
  \n(\sF^p , \sF^{m-p}) \ = \ 0 \,.
\end{equation}
Any flag $\sF^\sb$ satisfying \eqref{E:nu-iso-F} is called \emph{$\nu$--isotropic}.

Given $1 \le a < b \le m$ with $a+b \not=m+1$ and $a,b \not=r+1$, define
\begin{subequations} \label{SE:schubOG}
\begin{equation} \label{E:sOG1}
  X_{a,b}(\sF^\sb) \ \dfn \ \{ E \in \tOG(2,\bC^m) \ | \ 
  \tdim\,E\cap\sF^a \ge 1 \,,\ E \subset \sF^b \} \,\,.
\end{equation}
If $m = 2r+1$, then every Schubert variety of $\tOG(2,\bC^{2r+1})$ is $G$--congruent to one of \eqref{E:sOG1}.  If $m = 2r$ is even, define 
\[
  \tilde \sF^r \ = \ \tspan_\bC\{e_1,\ldots,e_{r-1}\,,\, e_{r+1} \} \,.
\]
Note that $\sF^1 \subset \cdots\subset\sF^{r-1} \subset \tilde\sF^r \subset \sF^{r+1}\subset \cdots$ is also a $\nu$--isotropic flag.  Set
\begin{eqnarray}
  \label{E:sOG2}
  \tilde X_{a,r}(\sF^\sb) & \dfn & \{ E \in \tOG(2,\bC^m) \ | \ 
  \tdim\,E\cap\sF^a \ge 1 \,,\ E \subset \tilde\sF^r \} 
  \,, \quad a \le r-1,\\
  \label{E:sOG3}
  \tilde X_{r,b}(\sF^\sb) & \dfn & \{ E \in \tOG(2,\bC^m) \ | \ 
  \tdim\,E\cap\tilde\sF^r \ge 1 \,,\ E \subset \sF^b \}
  \,,\quad b \ge r+2 \,.
\end{eqnarray}
\end{subequations}
Every Schubert variety of $\tOG(2,\bC^{2r})$ is $G$--congruent to one of \eqref{SE:schubOG}.
\end{example}

\subsection{The horizontal subbundle}

This section is a very brief review of the horizontal sub-bundle $T^1 \subset T(G/P)$, which may be characterized as the unique $G$--homogeneous, bracket--generating subbundle of the holomorphic tangent bundle. The reader interested in greater detail is encouraged to consult \cite{MR3217458} and the references therein.

We will need an representation theoretic description of the horizontal sub-bundle, and for this it will be convenient to recall the notion of a ``grading element.''

\subsubsection{Grading elements} \label{S:grelem}

Let $\{ \ttS^1 , \ldots , \ttS^r\}$ be the basis of $\fh$ dual to the simple roots.  A \emph{grading element} is any member of $\tspan_\bZ\{ \ttS^1,\ldots,\ttS^r\}$; these are precisely the elements $\ttT\in \fh$ of the Cartan subalgebra with the property that $\a(\ttT) \in \bZ$ for all roots $\a\in \Delta$.  Since $\ttT$ is semisimple (as an element of $\fh$), the Lie algebra $\fg$ admits a direct sum decomposition
\begin{subequations} \label{SE:grading}
\begin{equation}
  \fg \ = \ \bigoplus_{\ell\in\bZ} \fg^\ell
\end{equation}
into $\ttT$--eigenspaces
\begin{equation} \label{E:grT}
  \fg^\ell \ \dfn \ \{ \xi \in \fg \ | \ [\ttT,\xi] = \ell \xi \} \,.
\end{equation}
In terms of root spaces, we have
\begin{equation} \label{E:gr1} \renewcommand{\arraystretch}{1.3}
\begin{array}{rcl}
  \displaystyle \fg^\ell & = & 
  \displaystyle \bigoplus_{\a(\ttT)=\ell} \fg^\a \,,\quad \ell \not=0 \,,\\
  \displaystyle \fg^0 & = & 
  \displaystyle \fh \ \op \ \bigoplus_{\a(\ttT)=0} \fg^\a \,.
\end{array}
\end{equation}
\end{subequations}
The $\ttT$--eigenspace decomposition is a \emph{graded Lie algebra decomposition} in the sense that 
\begin{equation}\label{E:gr}
  \left[ \fg^\ell  , \fg^m \right] \ \subset \ \fg^{\ell+m} \,,
\end{equation}
a straightforward consequence of the Jacobi identity.  It follows that $\fg^0$ is a Lie subalgebra of $\fg$ (in fact, reductive), and each $\fg^\ell$ is a $\fg^0$--module.  The Lie algebra
\begin{equation} \label{E:p}
  \fp \,=\, \fp_\ttT \ = \ \fg^0 \,\op \,\fg^+ 
\end{equation}
is the \emph{parabolic subalgebra determined by the grading element $\ttT$}.  See \cite[Section 2.2]{MR3217458} for details.

\subsubsection{Minimal grading elements} \label{S:TvE}

Two distinct grading elements may determine the same parabolic $\fp$.  As a trivial example of this, given a grading element $\ttT$ and $0 < d \in \bZ$, both $\ttT$ and $d \ttT$ determine the same parabolic.  Among those grading elements $\ttT$ determining the same parabolic \eqref{E:p}, only one will have the property that $\fg^{\pm1}$ generates $\fg^\pm$ as an algebra.  That canonical grading element is defined as follows.  Given a parabolic $\fp$ subalgebra (and choice of Cartan and Borel $\fh \subset \fb \subset \fp$), the \emph{grading element associated to $\fp$} is
\begin{equation} \label{E:ttE}
  \ttE \ \dfn \ \sum_{i \in I(\fp)} \ttS^i \,.
\end{equation}
The reductive subalgebra $\fg^0 = \fg^0_\mathrm{ss} \op \fz$ has center $\fz = \tspan_\bC\{ \ttS^i \ | \ i \in I \}$ and semisimple subalgebra $\fg^0_\mathrm{ss} = [\fg^0,\fg^0]$.  A set of simple roots for $\fg^0_\mathrm{ss}$ is given by $\sS(\fg_0) = \{ \a_j \ | \ j \not \in I\}$.

\subsubsection{Definition} 

The (holomorphic) tangent bundle of $G/P$ is the $G$--homogeneous vector bundle 
\[
  T (G/P)\ = \ G \times_P (\fg/\fp) \,.
\]
Let $\ttE$ be the minimal grading element associated with $\fp$ (\S\ref{S:TvE}), and consider the associated grading \eqref{SE:grading}.  By \eqref{E:gr} and \eqref{E:p}, the quotient $\fg^{\ge-1}/\fp$ is a $\fp$--module.  Therefore,
\begin{equation} \label{E:ipr}
  T^1 \ \dfn \ G \times_P (\fg^{\ge-1}/\fp)
\end{equation}
defines a homogeneous, holomorphic subbundle of $T(G/P)$.  This subbundle is the \emph{horizontal subbundle}. We have $T^1 = T(G/P)$ if and only if $G/P$ is cominuscule. 

A \emph{horizontal submanifold} is an integral manifold of $T^1$; that is, a horizontal submanifold is any connected complex submanifold $M \subset G/P$ with the property that $T_zM \subset T^1_z$ for all $z \in M$, or any irreducible subvariety $Y \subset {G/P}$ with the property that $T_yY \subset T^1_y$ for all smooth points $y \in Y$.

\subsection{Horizontal Schubert varieties} \label{S:HSV}

The condition that the Schubert variety $X_w$ be horizontal is equivalent to $\Delta(w) \subset \Delta(\fg_1)$, where $\Delta(w)$ is given by \eqref{E:Dw}, cf.~\cite[Theorem 3.8]{MR3217458}.  A convenient way to test for this condition is as follows.  Let
$$
  \varrho \ \dfn \ \sum_{i=1}^r \w_i \ = \ \half \sum_{\a\in\Delta^+} \a
$$
be the sum of the fundamental weights (which is also half the sum of the positive roots).  Define
\begin{equation} \label{E:dfn_rho}
  \varrho_w \ \dfn \ \varrho \,-\, w(\varrho) \ = \ \sum_{\a\in\Delta(w)} \a \,.
\end{equation}
(See \cite[(5.10.1)]{MR0114875} for the second equality.)  Then 
$$
  |w| \ \le \ \varrho_w(\ttE) \ \in \ \bZ \,,
$$
and equality holds if and only if $\Delta(w) \subset \Delta(\fg_1)$; equivalently, $X_w$ is horizontal if and only if $\varrho_w(\ttE) = |w|$.  See \cite[Section 3]{MR3217458} for details.  Let
$$
  \sW_\mathrm{h} \ \dfn \ \{ w \in \sW^\fp_\tmin \ | \ \varrho_w(\ttE) = |w| \} 
$$
be the set indexing the Schubert variations of Hodge structure. 

Lemma \ref{L:sOG} and Corollary \ref{C:sOG} are continuations of Example \ref{eg:schub1}.

\begin{lemma}[Horizontal Schubert varieties in $\tOG(2,\bC^m)$] \label{L:sOG}
Recall the Schubert varieties \eqref{SE:schubOG} of $\tOG(2,\bC^m)$.
\begin{a_list_emph}
\item The Schubert variety $X_{a,b}(\sF^\sb)$ is horizontal if and only if $\sF^a \subset (\sF^b)^\perp$.
\item The Schubert varieties $\tilde X_{a,r}(\sF^\sb)$, are all horizontal.
\item None of the Schubert varieties $\tilde X_{r,b}(\sF^\sb)$ are horizontal. 
\end{a_list_emph}
\end{lemma}

\begin{corollary}[Maximal horizontal Schubert varieties in $\tOG(2,\bC^m)$] \label{C:sOG}
The maximal (with respect to containment) HSV in $\tSO(2,\bC^m)$, where $m \in \{2r,2r+1\}$, are: the $X_{a,m-a}(\sF^\sb)$, with $1 \le a \le r-1$; and
$\tilde X_{r-1,r}(\sF^\sb)$, if $m = 2r$.  Each of these maximal HSV is of dimension $m-4$.
\end{corollary}

\begin{proof}[Proof of Lemma \ref{L:sOG}]
By definition, a Schubert variety $X_w \subset \tOG(2,\bC^m)$ is horizontal (equivalently, satisfies Griffiths's transversality condition) if and only if 
\begin{equation} \label{E:MR3217458}
  \left. \td E \right|_{N_w \cdot o} \ \subset \ E^\perp \ \dfn \ 
  \{ v \in \bC^m \ | \ \n(E,v) = 0 \}
\end{equation}
on the Schubert cell $N_w \cdot o$.  \smallskip

(a) It is straight--forward to see that the condition $\sF^a \subset (\sF^b)^\perp$ implies $X_{a,b}(\sF^\sb)$ is horizontal.  Suppose that $\sF^a \not\subset (\sF^b)^\perp$; equivalently, $a+b > m$.  Then the condition $a+b \not=m+1$ (see Example \ref{eg:schub1}) forces $a+b \ge m+2$.  It follows that 
\[
  E(t) \ \dfn \ \tspan_\bC\{ e_b + t e_{m+1-a} \,,\,
  e_a - t e_{m+1-b} \} \ \in \ X_{a,b}(\sF^\sb) \,.
\]
(In fact, $E(t)$ lies in the Schubert cell.)  As \eqref{E:MR3217458} clearly fails for $E(t)$, we see that $X_{a,b}(\sF^\sb)$ is not horizontal.  This establishes assertion (a) of the lemma. \smallskip

(b)  As noted in Example \ref{eg:schub1}, the Schubert varieties $\tilde X_{a,r}(\sF^\sb)$ are defined for $a < r$.  Since $\sF^a \subset \tilde\sF^r = (\tilde\sF^r)^\perp$, it is immediate that the $\tilde X_{a,r}(\sF^\sb)$ are horizontal.  \smallskip

(c) Recall that the Schubert varieties $\tilde X_{r,b}(\sF^\sb)$ are defined for $r+2 \le b$, \cf Example \ref{eg:schub1}.  It follows that 
\[
  E(t) \ \dfn \ \tspan_\bC\{ e_b + t e_r \,,\,
  e_{r+1} - t e_{m+1-b} \} \ \in \ \tilde X_{r,b}(\sF^\sb) \,.
\]
(As above, $E(t)$ lies in the Schubert cell.)  As \eqref{E:MR3217458} clearly fails for $E(t)$, we see that $\tilde X_{r,b}(\sF^\sb)$ is not horizontal.  This establishes assertion (c) of the lemma.
\end{proof}

\section{Smooth horizontal Schubert varieties} \label{S:smooth}

A \emph{homogeneously embedded, homogeneous submanifold} of $G/P$ is a submanifold of the form 
$$
   g\,\{ aP \in G/P \ | \ a \in G' \} \ = \ g\, G'/P \,,
$$
where $g \in G$ and $G'$ is a closed Lie subgroup of $G$.  Such a submanifold is isomorphic to $G'/P'$, where $P' = G' \cap P$.  

Distinguished amongst the homogeneously embedded, homogeneous submanifolds of $G/P$ are the rational homogeneous subvarieties $X(\sD')$ corresponding to subdiagrams $\sD'$ of the Dynkin diagram $\sD$ of $\fg$.  The correspondence is as follows:  Identifying the nodes of $\sD$ with the simple roots $\sS$ of $\fg$, to any subdiagram we may associate a semisimple Lie subalgebra $\fg' \subset \fg$ by defining $\fg'$ to be the subalgebra generated by the root spaces $\{\fg^{\pm\a} \ | \ \a \in \sD' \}$.  Note that the Dynkin diagram of $\fg'$ is naturally identified with the subdiagram $\sD'$.  Let $G' \subset G$ be the corresponding closed, connected semisimple Lie subgroup.  For such $G'$, the subgroup $P' = G' \cap P$ is a parabolic subgroup.  The $G'$--orbit of $P \in G/P$ is isomorphic to $G'/P'$ and is the \emph{homogeneously embedded, rational homogeneous subvariety $X(\sD') \subset G/P$ corresponding to subdiagram $\sD' \subset \sD$}.

Define the index set $J = \{ j \ | \ \fg^{\a_j} \not\subset \fg' \}$.  Let $\ttF = \sum_{j \in J}\,\ttS^j$ be the corresponding grading element (Section \ref{S:grelem}), and let $\fg = \op \fg^\ell$ be the $\ttF$--eigenspace decomposition \eqref{SE:grading}.  Then $\fg' = [\fg^0 , \fg^0]$ is the semisimple component of the reductive subalgebra $\fg^0$.  Let $Q = P_J$ be the corresponding parabolic subgroup of $G$ (with Lie algebra $\fq = \fg^0 \op \fg^+$.)  It follows from the discussion of \cite[\S2.7.1]{MR2090671}\footnote{There is a typo in \cite[\S2.7.1]{MR2090671}; see \S\ref{S:tits_pt} for the corrected statement.} that
\begin{equation} \label{E:TvD}
  \hbox{$X(\sD') \subset G/P$ is the Tits transformation $\cT(Q/Q)$ of the point in 
  $Q/Q \in G/Q$.}
\end{equation}  
This, with \cite[Lemma 2.4]{MR3130568} (alternatively, see Lemma \ref{L:tits}), yields the following 

\begin{lemma} \label{L:subD}
The homogeneously embedded, rational homogeneous subvarieties $X(\sD') \subset G/P$ corresponding to subdiagrams $\sD' \subset \sD$ are smooth Schubert varieties of $G/P$.
\end{lemma}

\begin{remark} \label{R:Dw}
Let $\ttE$ be the grading element \eqref{E:ttE} associated with $\fp$.  It follows from the discussion above, that a Schubert variety $X_w$ is of the form $X(\sD')$ if and only if $\Delta(w) = \Delta(\fg'_+)$, where $\fg' = \op \fg'_\ell$ is the $\ttE$--eigenspace decomposition of $\fg'$ and $\Delta(w)$ is the set \eqref{E:Dw}.  By Section \ref{S:HSV}, $X(\sD')$ is a horizontal if and only if $\fg' = \fg'_1 \op \fg'_0 \op \fg'_{-1}$; equivalently, $X(\sD')$ is Hermitian symmetric.  
\end{remark}

The main result of the paper is 

\begin{mt}
Let $G/P$ be a rational homogeneous variety, and let $X \subset G/P$ be a horizontal Schubert variety.  If $X$ is smooth, then $X$ is a product of homogeneously embedded, rational homogeneous subvarieties $X(\sD') \subset G/P$ corresponding to subdiagrams $\sD' \subset \sD$.  Moreover, each $X(\sD')$ is cominuscule.
\end{mt}

\noindent   Since a homogeneously embedded Hermitian symmetric space is necessarily smooth, Theorem \ref{T:smooth} is equivalent to:  \emph{A HSV is smooth if and only if it is a horizontal homogeneously embedded Hermitian symmetric space.}  

Theorem \ref{T:smooth} is proved in Sections \ref{S:smprf}--\ref{S:prf4}.
\medskip

\begin{example}[Symplectic Grassmannians] \label{eg:SG}
In this example we will (a) illustrate the construction of $X(\sD')$ from $\sD'\subset \sD$, and (b) observe that not every smooth Schubert variety of a rational homogeneous variety is homogeneous; in particular, the assumption that $X$ be horizontal in Theorem \ref{T:smooth} is essential.

Let $\n$ be a non-degenerate, skew-symmetric bilinear form on $\bC^{2r}$.  Fix $\tti \le r$.  The symplectic grassmannian
\[
   \mathrm{SG}(\tti,\bC^{2r}) \ = \ 
   \left\{ E \in \tGr(\tti,\bC^{2r}) \ \left| \ \left.\n\right|_E = 0 \right.\right\}
\]
is a rational homogeneous variety $G/P$, where $G = \tAut(\bC^{2r},Q) = \tSp_{2r}\bC$ and $P = P_\tti$ is the maximal parabolic subgroup associated to the $\tti$--th simple root. The Dynkin diagram of $G$ (containing $r$ nodes) is 
{
 \setlength{\unitlength}{3pt}
 \begin{picture}(25,3)(0,-1)
 \multiput(0,0)(5,0){2}{\circle*{1}}
 \put(0,0){\line(1,0){5}}
 \multiput(15,0)(5,0){3}{\circle*{1}}
 \multiput(8,0)(2,0){3}{\circle*{0.4}}
 \put(15,0){\line(1,0){5}}
 \put(20,0.3){\line(1,0){5}}
 \put(20,-0.3){\line(1,0){5}}
 \put(22,-0.75){{\footnotesize{$\langle$}}}
 \end{picture}
}.
 
Fix a $Q$--isotropic flag $F^1 \subset F^2 \subset \cdots \subset F^{2r} = \bC^{2r}$; here $F^d$ is a complex linear subspace of dimension $d$ and $\n(F^d , F^{2r-d}) = 0$.  Given $0 \le a < \tti < b \le 2r-a$, 
$$
  X_{a,b} \ = \ \{ E \in \mathrm{SG}(\tti,\bC^{2r}) \ | \ F^a \subset E \subset F^b \}
$$
is a Schubert variety.  The variety $X_{a,b}$ is the homogeneous submanifold $X(\sD')$ associated to a Dynkin subdiagram $\sD' \subset \sD$ if and only if either $b \le r$, or $b = 2r-a$.  In the first case ($b \le r$), the subdiagram $\sD'$ corresponds to the simple roots $\sS' = \{ \a_{a+1},\ldots,\a_{b-1}\}$; we have $\fg' \simeq \fsl_{b-a}\bC$ and $X_{a,b} \simeq \tGr(\tti-a,\bC^{b-a})$; these Schubert varieties are horizontal.  In the second case ($b = 2r-a$), the subdiagram $\sD'$ corresponds to the simple roots $\sS' = \{ \a_{a-1},\ldots,\a_r\}$; we have $\fg'\simeq \fsp_{2(r-a)}\bC$ and $X_{a,b} = \mathrm{SG}(\tti-a,\bC^{2(r-a)})$; these $X_{a,b}$ are horizontal if and only if $\tti=r$ (equivalently, the symplectic grassmannian is a Hermitian symmetric Lagrangian grassmannian).

We illustrate this for $G/P=\mathrm{SG}(4,\bC^{16})$; that is, $\tti=4$ and $r=8$.  The encircled subdiagram 
\begin{center}
 \setlength{\unitlength}{3pt}
 \begin{picture}(35,3)(0,-1)
 \multiput(0,0)(5,0){8}{\circle*{1}}
 \put(15,0){\circle{1.9}}
 \put(0,0){\line(1,0){30}}
 \put(30,0.3){\line(1,0){5}}
 \put(30,-0.3){\line(1,0){5}}
 \put(32,-0.75){{\footnotesize{$\langle$}}}
 \put(15,0){\oval(25,4)}
 \end{picture}
\end{center}
corresponds to $\fg' = \fsl_6\bC$ and 
$$
  X_{1,7} \ = \ \{ E \in \mathrm{SG}(4,\bC^{16}) \ | \ F^1 \subset E \subset F^7 \} 
  \ = \ \tGr(3,F^7/F^1) \ \simeq \ \tGr(3,\bC^6) \,.
$$
Likewise, the encircled subdiagram 
\begin{center}
 \setlength{\unitlength}{3pt}
 \begin{picture}(35,3)(0,-1)
 \multiput(0,0)(5,0){8}{\circle*{1}}
 \put(15,0){\circle{1.9}}
 \put(0,0){\line(1,0){30}}
 \put(30,0.3){\line(1,0){5}}
 \put(30,-0.3){\line(1,0){5}}
 \put(32,-0.75){{\footnotesize{$\langle$}}}
 \put(22.5,0){\oval(30,4)}
 \end{picture}
\end{center}
corresponds to $\fg' = \fsp_{12}\bC$ and 
\[
  X_{2,14} \ = \ \{ E \in \mathrm{SG}(4,\bC^{16}) \ | \ F^2 \subset E \} 
  \ = \ \mathrm{SG}(2,(F^2)^\perp/F^2) \ \simeq \ \mathrm{SG}(2,\bC^{12}) \,.
\]

Returning to the general case, we have $\tti=r$ if and only if ${G/P}$ is a Hermitian symmetric space.  In this case, the $T^1 = T(G/P)$ so that every Schubert variety is a HSV.  Moreover, a Schubert variety is smooth if and only if it is the homogenous submanifold associated to a Dynkin subdiagram \cite{MR1703350}.

In contrast to the Hermitian symmetric case, if $\tti < r$ and $b = 2n-a-1$, then $X_{a,b}$ is smooth, but not homogeneous, cf.~\cite{MR2356323}.  These Schubert varieties are not horizontal.
\end{example}

\begin{remark}[Maximal parabolic associated to non--short root] \label{R:HM}
If $P$ is a maximal parabolic corresponding to a non--short simple root $\a_\tti$, then all smooth Schubert varieties of ${G/P}$ are homogeneously embedded, rational homogeneous varieties $X(\sD')$ corresponding to connected subdiagrams $\sD'\subset\sD$ containing the $\tti$--th node, cf.~\cite[Proposition 3.7]{HongMok2013}.  For example, the smooth Schubert varieties in the Grassmannian $\tGr(k,\bC^n)$ are all homogeneous.
\end{remark}

\subsection{Proof of Theorem \ref{T:smooth}: outline} \label{S:smprf}

The proposition is proved in three steps.  First we reduce to the case that $P$ is maximal (Section \ref{S:prf1}).  Then, the result \cite[Proposition 3.7]{HongMok2013} of Hong and Mok establishes the proposition in the case that the associated simple root is not short (Section \ref{S:prf2}).  Third, we address the short root case (Section \ref{S:prf3}).

\subsection{Reduce to the case that $P$ is maximal parabolic} \label{S:prf1}
Suppose that $X$ is a HSV.  Let $w \in W^\fp$ be the associated Weyl group element (Section \ref{S:schub}), so that $X = X_w = \overline{B w^{-1}o}$.  The condition that the Schubert variety be horizontal is equivalent to $\Delta(w) \subset \Delta(\fg_1)$ (Section \ref{S:HSV}).  By definition $\a \in\Delta(\fg_1)$ if and only if $\a(\ttE) = 1$.  Equivalently, $\a(\ttS^i) = 1$ for exactly one $i \in I$, and $\a(\ttS^j) = 0$ for all other $j \in I$.  Therefore, we have a disjoint union
$$
  \Delta(w) \ = \ \bigsqcup_{i \in I} \Delta_i(w)\,,
$$
where 
$$
  \Delta_i(w) \ \dfn \ \{ \a\in\Delta(w) \ | \ \a(\ttS^i) = 1 \} \,.
$$
It is straightforward to confirm that both $\Delta_i(w)$ and $\Delta^+ \backslash \Delta_i(w)$ are closed.  Therefore, $\Delta_i(w)$ determines a Schubert variety $X_i \subset G/P$, cf.~\cite[Remark 3.7(c)]{MR3217458}.  Since $\Delta_i(w) \subset \Delta(w)$ it follows from \cite[Lemma 8.1 \& Corollary 8.3]{MR3217458} that $X_i \subset X$.  Whence 
\[
  X \ = \ \prod_{i \in I} X_i \,.
\]

Let $\sD$ denote the Dynkin diagram of $G$.  Given $i \in I$, let $\sD \backslash \{ I \backslash\{i\}\}$ denote the subdiagram of $\sD$ obtained by removing all nodes corresponding to $j \in I \backslash\{i\}$ (and their adjacent edges).  Let $\sD_i$ denote the connected component of $\sD\backslash \{ I \backslash\{i\}\}$ containing the $i$-th node.  Let $G_i \subset G$ denote the closed, connected semisimple Lie subgroup of $G$ corresponding to $\sD_i$.  Let $\cO_i \subset G/P$ be the $G_i$--orbit of $o$.  Then $\cO_i \simeq G_i/(G_i\cap P)$ is rational homogeneous variety containing $X_i$ as a Schubert subvariety.  Moreover, $X_i$ is horizontal for the IPR on $\cO_i$.  Finally, note that $G_i\cap P$ is a maximal parabolic subgroup of $G_i$ -- the corresponding index set is just $\{i\}$.  Since $X$ is smooth if and only if each $X_i$ is smooth, to prove the proposition, it suffices to show that $X_i$ is a homogeneously embedded Hermitian symmetric space.  

This reduces the proof of the proposition to the case that $P$ is a maximal parabolic, which we now assume.  Let $\a_\tti$ denote the associated simple root.
 
\subsection{The case that $\a_\tti$ is not a short root} \label{S:prf2}
Suppose that $\a_\tti$ is not a short root of $G$.  By \cite[Proposition 3.7]{HongMok2013}, the Schubert variety $X \subset G/P_\tti$ is the homogeneously embedded, rational homogeneous subvariety corresponding to a subdiagram of $\sD$.  

\subsection{The case that $\a_\tti$ is a short root}\label{S:prf3}

In this case $G/P_\tti$ is one of the following five rational homogeneous varieties:
\begin{bcirclist}
\item 
Let $\n$ be a nondegenerate symmetric bilinear form on $\bC^{2r+1}$.  Then the \emph{orthogonal grassmannian} 
$$
  \mathrm{OG}(r,\bC^{2r+1}) \ = \ 
  \left\{ E \in \tGr(r,\bC^{2r+1}) \ \left| \ \left.\n\right|_E = 0 \right.\right\}
$$
of maximal, $\n$--isotropic subspaces is the rational homogeneous variety $B_r/P_r = \tSpin_{2r+1}\bC/P_r$, where $P_r$ is the maximal parabolic associated to the short simple root $\a_r$.
\item
Let $\n$ be a nondegenerate skew-symmetric bilinear form on $\bC^{2r}$.  Then the \emph{symplectic grassmannian} 
$$
  \mathrm{SG}(\tti,\bC^{2r}) \ = \ 
  \left\{ E \in \tGr(r,\bC^{2r+1}) \ \left| \ \left.\n\right|_E = 0 \right.\right\}
$$
of $\tti$--dimensional $\n$--isotropic subspaces is the rational homogeneous variety $C_r/P_\tti = \tSp_{2r}\bC/P_\tti$, where $P_\tti$ is the maximal parabolic associated to the simple root $\a_\tti$.  The simple root is short if and only if $\tti < r$.
\item The exceptional $F_4/P_3$, $F_4/P_4$ or $G_2/P_1$.
\end{bcirclist}
Unfortunately, the Hong--Mok argument does not appear to extend to $\mathrm{SG}(\tti,\bC^{2r})$, $F_4/P_3$ or $G_2/P_1$.  Instead, we will see that, for each of the five ${G/P}$ above, Theorem \ref{T:smooth} follows from either (i) the Brion--Polo classification of smooth minuscule Schubert varieties, or (ii) existing descriptions of the HSV, or a combination of both.

\subsubsection{The case that ${G/P} = \mathrm{OG}(r,\bC^{2r+1})$}

In this case, ${G/P}$ is minuscule, cf.~\cite{MR1782635}.  Brion and Polo have shown that the smooth Schubert varieties of minuscule $G/P$ are homogeneous submanifolds.  More precisely, let $Q \supset B$ denote the stabilizer of $X$, cf.~Remark \ref{R:schubStab}.  Then $X = Q\cdot o$ by \cite[Proposition 3.3(a)]{MR1703350}.  

Let $\fq = \fq^0 \op \fq^+$ be the graded decomposition \eqref{E:p} associated to associated to the parabolic $\fq$ (and choices $\fq \supset \fb \supset \fh$).  If $Q^0 \subset Q$ is the closed Lie subgroup with Lie algebra $\fq^0$, then $Q \simeq Q^0 \times \fq^+ = Q^0 \times \texp(\fq^+)$, by \cite[Theorem 3.1.3]{MR2532439}.  Therefore, $X = Q \cdot o = Q^0 \cdot o = Q^0_\mathrm{ss}\cdot o$, where $Q^0_\mathrm{ss}$ is the closed semisimple Lie subgroup with Lie algebra $\fq^0_\mathrm{ss} = [\fq^0,\fq^0]$.  The semisimple $Q^0_\mathrm{ss}$ is the subgroup of $G$ associated to the subdiagram $\sD' = \{ \a \in \sD \ | \ \fg^\a \in \fq^0 \}$ (Section \ref{S:grelem}).  Thus, $X = X(\sD')$ is the homogeneously embedded, rational homogeneous subvariety associated to a subdiagram $\sD'\subset\sD$.

\subsubsection{The case that ${G/P} = \mathrm{SG}(\tti,\bC^{2r})$}

Recall that $\tti < r$.  In this case, the marked Dynkin diagram associated to $\mathrm{SG}(\tti,\bC^{2r})$ is 
\begin{center}
\setlength{\unitlength}{3pt}
\begin{picture}(45,2)(0,-1)
 \multiput(0,0)(5,0){2}{\circle*{1}}
 \put(0,0){\line(1,0){5}}
 \multiput(8,0)(2,0){3}{\circle*{0.4}}
 \multiput(15,0)(5,0){3}{\circle*{1}}
 \put(15,0){\line(1,0){10}}
 \multiput(28,0)(2,0){3}{\circle*{0.4}}
 \multiput(35,0)(5,0){3}{\circle*{1}}
 \put(35,0){\line(1,0){5}}
 \put(40,0.3){\line(1,0){5}}
 \put(40,-0.3){\line(1,0){5}}
 \put(42,-0.675){{\scriptsize{$\langle$}}}
 \put(-0.8,1.5){\footnotesize{1}}
 \put(4.2,1.5){\footnotesize{2}}
 \put(19.2,1.5){\footnotesize{$\tti$}}
 \put(44.2,1.5){\footnotesize{$r$}}
 \put(20,0){\circle{1.9}}
\end{picture}
\end{center}
It follows from the proof of \cite[Proposition 3.11]{MR3217458}\footnote{See, in particular, Step 3 of Section 7.3 in \cite{MR3217458}; in the case under consideration, $t = 1$ and $i_t = \tti$.} that there is a unique maximal Schubert variation of Hodge structure; it is the homogeneous submanifold $Z' \simeq \tGr(\tti,\bC^{r})$ associated to the circled subdiagram 
\begin{center}
\setlength{\unitlength}{3pt}
\begin{picture}(45,2)(0,-1)
 \multiput(0,0)(5,0){2}{\circle*{1}}
 \put(0,0){\line(1,0){5}}
 \multiput(8,0)(2,0){3}{\circle*{0.4}}
 \multiput(15,0)(5,0){3}{\circle*{1}}
 \put(15,0){\line(1,0){10}}
 \multiput(28,0)(2,0){3}{\circle*{0.4}}
 \multiput(35,0)(5,0){3}{\circle*{1}}
 \put(35,0){\line(1,0){5}}
 \put(40,0.3){\line(1,0){5}}
 \put(40,-0.3){\line(1,0){5}}
 \put(42,-0.675){{\scriptsize{$\langle$}}}
 \put(20,0){\circle{1.9}}
 \put(20,0){\oval(43,3)}
\end{picture}
\end{center}
So, any smooth HSV in $\mathrm{SG}(\tti,\bC^{2r})$ is a smooth Schubert variety of $\tGr(\tti,\bC^r)$.  These are well-known to be precisely the homogeneously embedded, rational homogeneous subvarieties corresponding to connected subdiagrams of the the marked Dynkin diagram for $\tGr(\tti,\bC^r)$ that contain the $\tti$--th node, cf.~\cite[Section 9.3]{MR1782635}.  (Alternatively, this follows from the Brion--Polo result, since $\tGr(\tti,\bC^r)$ is minuscule, or Remark \ref{R:HM}.)  

\subsubsection{The case that ${G/P} = F_4/P_3$}

By \cite[Example 5.16]{MR3217458}, the only Schubert variations of Hodge structure are (the trivial $o \in G/P$ and) the $\bP^1 \subset \bP^2$ corresponding to the two subdiagrams
\begin{center}
 \setlength{\unitlength}{4pt}
 \begin{picture}(15,1)(0,-1)
 \multiput(0,0)(5,0){4}{\circle*{1}}
 \put(0,0){\line(1,0){5}}
 \put(5,0.3){\line(1,0){5}}
 \put(5,-0.3){\line(1,0){5}}
 \put(10,0){\line(1,0){5}}
 \put(6.5,-0.6){{\footnotesize{$\rangle$}}}
 \put(10,0){\circle{1.9}}
 \put(10,0){\oval(4.5,3)}
 \end{picture}
 \hsp{40pt}
 \begin{picture}(15,1)(0,-1)
 \multiput(0,0)(5,0){4}{\circle*{1}}
 \put(0,0){\line(1,0){5}}
 \put(5,0.3){\line(1,0){5}}
 \put(5,-0.3){\line(1,0){5}}
 \put(10,0){\line(1,0){5}}
 \put(6.5,-0.6){{\footnotesize{$\rangle$}}}
 \put(10,0){\circle{1.9}}
 \put(12.5,0){\oval(9,3)}
 \end{picture}
\end{center}
Therefore, any smooth HSV in $F_4/P_3$ is a homogeneously embedded, rational homogeneous subvariety corresponding to a subdiagram of $\sD$.  

\subsubsection{The case that $G/P = F_4/P_4$}
By \cite[Example 5.17]{MR3217458}, the only Schubert variations of Hodge structure are (the trivial $o \in {G/P}$ and) the $\bP^1 \subset \bP^2$ corresponding to the two subdiagrams
\begin{center}
 \setlength{\unitlength}{4pt}
 \begin{picture}(15,1)(0,-1)
 \multiput(0,0)(5,0){4}{\circle*{1}}
 \put(0,0){\line(1,0){5}}
 \put(5,0.3){\line(1,0){5}}
 \put(5,-0.3){\line(1,0){5}}
 \put(10,0){\line(1,0){5}}
 \put(7,-0.6){{\footnotesize{$\rangle$}}}
 \put(15,0){\circle{1.9}}
 \put(15,0){\oval(4.5,3)}
 \end{picture}
 \hsp{40pt}
 \begin{picture}(15,1)(0,-1)
 \multiput(0,0)(5,0){4}{\circle*{1}}
 \put(0,0){\line(1,0){5}}
 \put(5,0.3){\line(1,0){5}}
 \put(5,-0.3){\line(1,0){5}}
 \put(10,0){\line(1,0){5}}
 \put(6.5,-0.6){{\footnotesize{$\rangle$}}}
 \put(15,0){\circle{1.9}}
 \put(12.5,0){\oval(9,3)}
 \end{picture}
\end{center}
Therefore, any smooth HSV in $F_4/P_4$ is a homogeneously embedded, rational homogeneous subvariety corresponding to a subdiagram of $\sD$.  
\subsubsection{The case ${G/P} = G_2/P_1$}
This variety is the quadric hypersurface $\cQ^5 \subset \bP^6$.  By \cite[Example 5.30]{MR3217458}, the only HSV are (the trivial $o \in {G/P}$ and) the smooth $\bP^1 \subset {G/P}$ which is the homogeneous submanifold corresponding to the circled subdiagram
\begin{center}
 \setlength{\unitlength}{4pt}
 \begin{picture}(15,1)(0,-1)
 \multiput(0,0)(5,0){2}{\circle*{1}}
 \put(0,0){\line(1,0){5}}
 \put(0,0.3){\line(1,0){5}}
 \put(0,-0.3){\line(1,0){5}}
 \put(2.5,-0.6){{\footnotesize{$\langle$}}}
 \put(0,0){\circle{1.7}}
 \put(0,0){\oval(4,3)}
 \end{picture}
\end{center}
Therefore, any smooth HSV in $G_2/P_1$ is a homogeneously embedded, rational homogeneous subvariety corresponding to a subdiagram of $\sD$.  

\subsection{Fini} \label{S:prf4}

We have shown, in Sections \ref{S:prf1}--\ref{S:prf3}, that given any rational homogeneous variety $G/P$ (here the parabolic is arbitrary -- $P$ need not be maximal) any smooth HSV $X\subset {G/P}$ is a product homogeneously embedded, rational homogeneous subvarieties $X_i = X(\sD_i)$ corresponding to a subdiagram of $\sD_i \subset \sD$.  The condition that $X$ (and therefore each $X_i$) be horizontal forces $X_i$ to be Hermitian symmetric.  This completes the proof of Theorem \ref{T:smooth}.  

\begin{remark}
It is possible that \cite[Theorem 2.6]{MR1703350} can be used to prove Theorem \ref{T:smooth}, as it was used to establish the homogeneity of minuscule and cominuscule Schubert varieties in \cite{MR1703350}.  The arguments of \cite{MR1703350} are representation theoretic in nature; in the proof of Theorem \ref{T:smooth}, we elected for the more geometric approach of \cite{HongMok2013}.
\end{remark}

\section{Lines on flag varieties} \label{S:lines}

Throughout this section we 
\begin{quote}
\emph{assume $P$ is a maximal parabolic subgroup, and identify $G/P$ with its image under the minimal homogeneous embedding $G/P \inj \bP V$}
\end{quote}
of Section \ref{S:rhv}.  The main result of this section is Proposition \ref{P:Xo} which characterizes the flag manifolds $G/P$ with the property that the Schubert variety swept out by the horizontal lines through a point is itself horizontal.  The proposition is applied in \cite{KR1} to study 

  Fix a highest weight vector $0\not=v \in V$, so that $[v] = o \in \bP V$ is the highest weight line.  By \eqref{E:p}, the tangent space $T_o(G/P)$ is naturally identified with $\fg/\fp$ as a $\fp$--module, and with $\fg^-$ as a $\fg^0$--module; for the most part, we will work with the latter identification.  The set of embedded, linear $\bP^1 \subset \bP V$ containing $o$ and tangent to $G/P$ at that point is in bijection with $\bP \fg^- = \bP\,T_o(G/P)$.  To be precise, given a tangent line $[\xi] \in \bP \fg^-$, we have 
$$
  \bP^1 \ = \ \bP^1(o,[\xi]) \ \dfn\ \bP \, \tspan_\bC\{v , \xi(v)\} \subset \bP V \,.
$$  
Making use of this identification, let 
$$
  \tilde \cC_o \ \dfn \ \{ [\xi] \in \bP \,\fg^- \ | \ \bP^1(o,[\xi]) \subset G/P \}
  \ = \ \{ \bP^1 \subset \bP V \ | \ o \in \bP^1 \subset G/P \} 
$$ 
be the set of lines on $G/P$ passing through $o$.  (For a general embedding $G/P \inj \bP V$, not necessarily minimal, $\tilde \cC_o$ is defined to be the variety of minimal rational tangents, cf.~\cite{MR1748609}.)  The subvariety of lines tangent to $\fg^{-1} \subset \fg^- \simeq T_o(G/P)$ is
\begin{equation}\label{E:Co}
  \cC_o \ \dfn \ \{ [\xi] \in \tilde \cC_o \ | \ \xi \in \fg^{-1} \} 
  \ = \ \{ \bP^1 \subset (G/P) \ | \ \bP^1 \ni o 
  \hbox{ is horizontal}\}\,.
\end{equation}

Let 
\begin{equation}\label{E:Xo}
  X \ \dfn \ \bigcup_{\bP^1 \in \cC_o} \bP^1
\end{equation}
be the variety swept out by the lines that pass through $o$ and are horizontal.

\begin{proposition} \label{P:Xo}
Assume that $P$ is a maximal parabolic subgroup of $G$.  Then 
\begin{a_list_emph}
\item $X$ is a cone over $\cC_o$ with vertex $o$ and a Schubert variety.  
\item $X$ is horizontal if and only if the simple root $\a_\tti$ associated with the maximal parabolic $\fp$ is not short.
\end{a_list_emph}
\end{proposition}

\noindent The proposition is proved in Section \ref{S:linespt}.

\subsection{Properties of $\cC_0$} \label{S:Pmax}

We now recall two properties of $\cC_o$ in the case that $P$ is maximal (equivalently, $I = \{\tti\}$ and $\ttE = \ttS^\tti$).  The results that follow are due to \cite{MR1966752}, where $\cC_o$ and $\tilde\cC_o$ are discussed for arbitrary (not necessarily maximal) $P$.
\begin{a_list_bold}
\item  
   Since $P$ is maximal, $\fg^{-1}$ is an irreducible $\fg^0$--module with highest weight line $\fg^{-\a_\tti}$, cf.  \cite[Theorem 8.13.3]{MR928600}.  The variety of lines $\cC_o \subset \bP \fg^{-1}$ is the $G^0$--orbit of this highest weight line, cf.~\cite[Theorem 4.3]{MR1966752}.  In particular, $\cC_o$ is a rational homogeneous variety; indeed, 
$$
  \cC_o \ \simeq \ G^0/(G^0\cap Q)\,,
$$ 
where $Q \supset B$ is the parabolic subgroup defined by 
\begin{equation} \label{E:Iq}
  I(\fq) \ = \ \{ j \ | \ \fg^{-\a_j} \not\subset \fq\} 
  \ \dfn \ \{ j \ | \ \langle \a_\tti , \a_j \rangle \not=0\} \,.
\end{equation}
That is, the simple roots indexed by $I(\fq)$ are those adjacent to $\a_\tti$ in the Dynkin diagram of $\fg$; cf.~\cite[Proposition 2.5]{MR1966752}.  With only a few exceptions, $\cC_o$ is a $G_0$--cominuscule variety; equivalently, $\cC_o \simeq G^0/(G^0\cap Q)$ admits the structure of a compact Hermitian symmetric space, cf.~\cite[Proposition 2.11]{MR1966752}.  
\item
If the simple root $\a_\tti$ associated with the maximal parabolic $P$ is not short, then $\cC_o = \tilde \cC_o$, cf.~\cite[Theorem 4.8.1]{MR1966752}.  If the simple root is short, then $\tilde \cC_o$ is the union of two $P$--orbits, and open orbit and its boundary $\cC_o$.  
\end{a_list_bold}

\subsection{Uniruling of $G/P$} \label{S:uniruling}

We continue with the assumption that $P$ is maximal; however, analogous statements follow for unirulings on general $G/P$.  For the more general statements, it is convenient to use Tits correspondences, which are briefly reviewed in Section \ref{S:tits}.

Given $x = g o \in G/P$, with $g \in G$, let $\cC_x = g \cC_o$ denote the corresponding set of lines through $x$. (It is an exercise to show that $\cC_x$ is well-defined; that is, $\cC_x$ does not depend on our choice of $g$ yielding $x = go$.)  Then 
$$
  \cC \ \dfn \ \{ \bP^1 \ | \ \bP^1 \in \cC_x \,,\ x\in G/P\} 
  \ = \ \bigcup_{g \in G} g \,\cC_o \,.
$$
forms a uniruling of $G/P$.  (As will be shown in Corollary \ref{C:pt}, this uniruling is parameterized by $G/Q$ -- that is, $\cC \simeq G/Q$.)  

\begin{remark}\label{R:tC}
More generally, the set of \emph{all} lines on $G/P$ is $\tilde \cC = \cup_{g\in G} \, g \,\tilde\cC_o$.
\begin{a_list_emph}
\item  It follows from definition \eqref{E:Co} and the homogeneity of the IPR, that 
\begin{equation} \label{E:Cogen}
  \cC \ = \ \{ \bP^1 \in \tilde\cC \ | \ \bP^1 \hbox{ is horizontal}\}
\end{equation}
is precisely the set of lines on $G/P$ that are integrals of the IPR.
\item  As noted in Section \ref{S:Pmax}(b), if the simple root associated to the maximal parabolic $P$ is not short, then $\tilde \cC = \cC$ consists of a single $G$--orbit.  If the simple root is short, then $\tilde\cC$ consists of two $G$--orbits, an open orbit and its boundary $\cC$, cf.~\cite[Theorem 4.3]{MR1966752}.
\end{a_list_emph}
\end{remark}

\subsection{Tits correspondences} \label{S:tits} 

Tits correspondences describe homogeneous unirulings of a rational homogeneous variety $G/P$ by homogeneously embedded, rational homogeneous subvarieties $G'/P'$; these unirulings may be used to clarify the geometry of $G/P$.  The material in Sections \ref{S:tits_schub}--\ref{S:linespt} is taken from \cite{MR3130568, CR2}.  Given two standard parabolics $P$ and $Q=P_J$, the intersection $P \cap Q$ is also a standard parabolic.  (Note that $I(\fp\cap\fq) = I(\fp) \cup I(\fq)$.)  There is a natural double fibration, called the \emph{Tits correspondence}, given by the diagram in Figure \ref{f:tits};  here the maps $\eta$ and $\tau$ are the natural projections. 
\begin{figure}[h]
\caption{Tits correspondence}
\setlength{\unitlength}{4pt}
\begin{picture}(33,12)(0,0)
\put(0,0){$G/P$} 
\put(10,6.8){\vector(-1,-1){4}}
\put(10.5,8.5){$G/(P\cap Q)$}
\put(23,6.8){\vector(1,-1){4}}
\put(28,0){$G/Q$}
\put(6,5.5){$\eta$}
\put(25.5,5){$\tau$}
\end{picture}
\label{f:tits}
\end{figure}
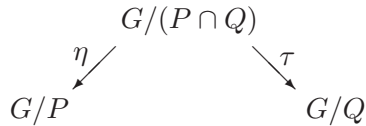
Given a subset $\Sigma\subset G/Q$, the \emph{Tits transform} is $\mathcal{T}(\Sigma) : = \eta(\tau^{-1}(\Sigma))$.  

\subsection{Tits transform of a point} \label{S:tits_pt}

The Tits transform of a point $y \in G/Q$ will play a crucial r\^ole in our discussion of $G/P$ and its Schubert varieties.   What follows is a brief review of \cite[\S2.7.1]{MR2090671}.\footnote{There is a typo in \cite[\S2.7.1]{MR2090671}; on page 87, $S$ and $S'$ should be swapped.  This is corrected for in the discussion above.}

The Tits transform $\cT(Q/Q)$ of the point $Q/Q \in G/Q$ is the $G'$--orbit $G'(P/P) \subset G/P$, where $G'$ is the semisimple subgroup of $G$ whose Dynkin diagram $\sD'$ is obtained from the diagram $\sD$ of $G$ by removing the nodes $I(\fq)\backslash I(\fp)$.  Therefore, $\cT(Q/Q) \simeq G'/P'$, where $P' = G' \cap P$ is the parabolic subgroup of $G'$ with index set \eqref{E:I} given by $I(\fp') = I(\fp) \backslash I(\fq)$.  

Since any point $y \in G/Q$ is of the form $y = g Q/Q$ for some $g \in G$, and $\cT(g Q/Q) = g\cT(Q/Q)$, it follows that $\cT(y) \simeq G'/P'$ and 
\begin{subequations} \label{SE:tits}
\begin{equation} \label{E:titsa}
\begin{array}{c}
\hbox{\emph{$G/P$ is uniruled by subvarieties $\Sigma$ isomorphic to $G'/P'$,}} \\ 
\hbox{\emph{and the uniruling is parameterized by $G/Q$.}}
\end{array}
\end{equation}
Moreover, $G/(P\cap Q)$ is the incidence space for this uniruling.  Precisely, 
\begin{equation} \label{E:incidence}
  G/(P\cap Q) \ = \ \{ (x,\Sigma) \in G/P \times G/Q \ | \ x \in \Sigma \}\,.
\end{equation}
\end{subequations}

\begin{remark}[$\bP^1$--unirulings of $G/P$ for maximal $P$] \label{R:P1}
In the case that $P \subset G$ is a maximal parabolic (Section \ref{S:rhv}), there is a unique $G$--homogeneous variety $G/Q$ parameterizing a uniruling of $G/P$ by lines $\bP^1$; it may be identified by inspection of the Dynkin diagram $\sD$ of $\fg$ as follows.  The maximality of $P$ is equivalent to $I(\fp) = \{\tti\}$ for some $\tti$.  In order to obtain a uniruling by $\bP^1$s, we must choose $Q$ (equivalently, the index set $I(\fq)$) so that $G'/P' = \bP^1$.  To that end, let $J = \{ j \not= \tti \ | \ (\a_\tti,\a_j) \not=0 \}$ index the nodes in the Dynkin diagram that are adjacent to the $\tti$--th node.  Then $G'/P' \simeq \bP^1$ (equivalently, $G/Q$ parameterizes a uniruling of $G/P$ by $\bP^1$s) if and only if $J \subset I(\fq)$.  When $I(\fq) = J$ we say that \emph{$G/Q$ is the smallest rational $G$--homogeneous variety parameterizing a uniruling of $G/P$ by lines $\bP^1$}.
\end{remark}

\subsection{Tits transform of a Schubert variety} \label{S:tits_schub}

The Tits transform preserves Schubert varieties; that is, if $Y \subset G/Q$ is a Schubert variety, then so is $X = \cT(Y) = G/P$.  Moreover, $X$ may be determined as follows.  Let $\sW^\fq_\tmin$ be the set of (unique) minimal length representatives of the right cosets $\sW_\fq \backslash \sW$, and let $\sW^\fq_\mathrm{max} \subset \sW$ denote this set of (unique) maximal length representatives.  If $w_0$ is the longest element of $\sW_\fq$, then $\sW^\fq_\tmax = \{ w_0 w \ | \ w \in \sW^\fq_\tmin\}$.  We have $\sW_\fq \backslash \sW \simeq \sW^\fq_\tmin \simeq \sW^\fq_\tmax$.  A proof of the following well--known lemma is given in \cite{MR3130568}.

\begin{lemma} \label{L:tits}
Let $w \in \sW^\fq_\tmax$, and let $Y_w \subset G/Q$ denote the Schubert variety indexed by the coset $\sW_\fq w$.  Then the Tits transformation $\cT(Y_w)$ is the Schubert variety $X_w \subset G/P$ indexed by the coset $\sW_\fp w$.
\end{lemma}

\subsection{Lines through a point} \label{S:linespt}

The following, first observed in \cite{CR2}, is an immediate consequence of \eqref{SE:tits}.

\begin{lemma} \label{L:pt}
Suppose that $G/Q$ parameterizes a uniruling of $G/P$ by $\bP^1$s.  Let $\Sigma = \tau(\eta^{-1}(o)) \subset G/Q$ denote the image of $o \in G/P$ under the Tits transform, and let $\cT(\Sigma) = \eta(\tau^{-1}(\Sigma)) \subset G/P$ denote the image of $\Sigma$ under the Tits transform back to $G/P$.  Then $\Sigma$ is the set of lines in $G/P$ that are parameterized by $G/Q$ and pass through $o \in G/P$.  Likewise, 
$$
  \cT(\Sigma) \ = \ \bigcup_{o \in \bP^1 \in G/Q} \bP^1 
$$
is the subset of $G/P$ swept out by these lines, and is naturally identified with a cone $\cC(\Sigma)$ over $\Sigma$ with vertex $o$.
\end{lemma}

\begin{remark} \label{R:pt}
By Lemma \ref{L:tits}, the variety $\cT(\Sigma)$ is a Schubert variety $X_w \subset G/P$.
\end{remark}

\begin{corollary} \label{C:pt}
Let $P$ be a maximal parabolic and let $Q$ be the parabolic of Section \ref{S:Pmax}(a).  Then  the uniruling $\cC$ of $G/P$ (Section \ref{S:uniruling}) is is precisely the uniruling $G/Q$ obtained through the Tits correspondence.  In particular, the variety $\Sigma$ of Lemma \ref{L:pt} is the variety $\cC_o$ of \eqref{E:Co}, and the variety $\cT(\Sigma)$ of Lemma \ref{L:pt} is the variety $X$ of \eqref{E:Xo}.  
\end{corollary}

\begin{proof}
By the maximality of $P$ we have $I(\fp) = \{\tti\}$, \cf Section \ref{S:rhv}.  Recall the Levi subgroup $G^0 \subset G$ whose Lie algebra $\fg^0$ is the zero--eigenspace of the grading element $\ttE = \ttS^\tti$ associated with the parabolic $\fp$, \cf Section \ref{S:rhv}.  Then the simple roots of the semisimple Lie algebra $\fg^0_\tss = [\fg^0,\fg^0]$ are $\sS \backslash \{\a_\tti\}$, the simple roots of $\fg$ minus the $\tti$--th.   As discussed in Section \ref{S:Pmax}(a), we have $\cC_o = G^0 /(G^0 \cap Q)$.  It follows that $\cC_o = G^0_\mathrm{ss} /(G^0_\mathrm{ss} \cap Q)$, where $G^0_\tss \subset G^0$ is the semisimple subgroup with Lie algebra $\fg^0_\tss$. Note the index set $I(\fg^0_\mathrm{ss} \cap \fq)$ associated with the parabolic $G^0_\mathrm{ss} \cap Q$ by \eqref{E:I} is $I(\fq)$.

As discussed in Section \ref{S:Pmax}(a), the set $I(\fq)$ indexes those nodes of the Dynkin diagram that are are adjacent to the $\tti$--th node.  By Remark \ref{R:P1}, $G/Q$ is the minimal rational $G$--homogeneous variety parameterizing a uniruling of $G/P$ by lines $\bP^1$s.   The Tits transform $\Sigma = \cT(P/P) \subset G/Q$ is of the form $G'/P'$.  From the descriptions of $G'$ and $P'$ in Section \ref{S:tits_pt} and the discussion of Remark \ref{R:P1} we see that $G' = G^0_\mathrm{ss}$ and $P' = G^0_\mathrm{ss} \cap Q$.  Thus, $\cC_o = \Sigma$.
\end{proof}

\begin{proof}[Proof of Proposition \ref{P:Xo}]
Part (a) of the lemma follows directly from Lemma \ref{L:pt} and Corollary \ref{C:pt}.  The argument for part (b) is based on the Tits transform recipe given by Lemma \ref{L:tits} and the characterization of horizontal Schubert varieties \cite{KR2}.  We will make use (without explicit mention) of observations made in the proof of Corollary \ref{C:pt}.

To begin, let $P \subset G$ be a maximal parabolic with index set $I(\fp) = \{\tti\}$, \cf Section \ref{S:rhv}.  Let $G/Q$ parameterize the uniruling of $G/P$ by $\bP^1$s.  The right coset indexing the Schubert variety $o = P/P \in G/P22$ is $\sW_\fp w_0$, where $w_0$ is the longest word of the Weyl group $\sW_\fp$ of $\fg^0$.  By Lemma \ref{L:tits}, the Schubert variety $\Sigma = \cT(P/P)$ is indexed by the coset $\sW_\fq w_0$.  Let $w_1 \in \sW^\fq_\tmax$ be the longest representative of $\sW_\fq w_0 = \sW_\fq w_1$.  Again, Lemma \ref{L:tits} implies $\cT(\Sigma)$ is the Schubert variety indexed by the right coset $\sW_\fp w_1$.  By Corollary \ref{C:pt}, this is the Schubert variety $X$.  Let $w \in \sW^\fp_\tmin$ be a shortest representative of $\sW_\fp w_1 = \sW_\fp w$.  As discussed in Section \ref{S:HSV}, the Schubert variety $X$ is horizontal if and only if $\varrho_w(\ttE) = \tdim\,X$.  So the substance of the proof is to compute the integer $\varrho_w(\ttE)$.  Note that $\ttE = \ttS^\tti$, \cf\eqref{E:ttE}.

Let $v \in \sW_\tmin^\fq$ be the shortest Weyl element indexing the Schubert variety $\Sigma$.  Let $w_0' \in \sW_\fq$ be the longest element of the Weyl subgroup, so that $w_0' v = w_1$ is the longest element indexing $\Sigma$.   Then
$$
  \sW_\fq w_0 \ = \ \sW_\fq v \ = \ \sW_\fq w_0'v\,,
$$
and $\sW_\fp w_0' v$ indexes $X = \cT(\Sigma)$.  We claim that 
\begin{equation} \label{E:c1}
  \sW_\fp w_0' \ = \ \sW_\fp (\tti) \,,
\end{equation}
where $(\tti) \in \sW$ denotes the reflection associated with the simple root $\a_\tti$.  It follows from \eqref{E:c1} that 
$$
  (\tti) v \ = \ w \ \in \ \sW^\fp_\tmin
$$
is the shortest Weyl group element indexing  $X$.  

\begin{proof}[Proof of \eqref{E:c1}]
Observe that $\sW_\fp$ is generated by the simple reflections $\{ (j) \ | \ j \not=\tti\}$.  Likewise, $\sW_\fq$ is generated by the simple reflections $\{ (j) \ | \ j \not= \tti-1,\tti+1\}$.  In particular, any element $u$ of $\sW_\fq$ may be written as $u' (\tti)$ with $u' \in \sW_\fp$.  The claim follows.
\end{proof}

We return to the proof of Proposition \ref{P:Xo}(b).  From the discussion of Section \ref{S:schub} (specifically \eqref{E:len=dim}) we see that 
$$
  |v| \ = \ |\Delta(v)| \ = \ \tdim\,\Sigma \tand
  1 \,+\, |v| \ = \ |w| \ = \ |\Delta( w )| \ = \ \tdim\,X \,.
$$
The argument establishing \cite[Proposition 3.2.14(5)]{MR2532439} yields
$$
  \Delta(w) \ = \ \{ \a_\tti \} \ \cup \ (\tti) \Delta(v) \,.
$$
Since $v$ indexes a \emph{homogeneous} Schubert variety $\Sigma = G'/P'$, we see from the discussion of Section \ref{S:schub} that 
$$
  \Delta(v) \ = \ \{ \a \in \Delta(\fg^0) \ | \ \a(\ttF) > 0 \}\,,
$$
where 
$$
  \ttF \ = \ \sum_{i \in I(\fq)} \ttS^i
$$
is the grading element \eqref{E:ttE} associated with the parabolic $\fq$.  Since $\fg^0$ is the zero--eigenspace of $\ttE = \ttS^\tti$, this may be rewritten as 
$$
  \Delta(v) \ = \ \{ \a \in \Delta \ | \ \a(\ttE) = 0 \,,\ \a(\ttF) > 0 \} \,.
$$
In particular, the roots of $\Delta(v)$ are precisely those positive roots $\a = m^i\a_i$ such that $m^\tti=0$ and $m^j > 0$ for some $j \in I(\fq)$.  Informally, the root $\a$ does not involve the simple root $\a_\tti$, but does involve some simple root adjacent to $\a_\tti$.  Whence the reflection
$$
  (\tti)\a \ = \ \a \ - \ 2\frac{(\a,\a_\tti)}{(\a_\tti,\a_\tti)} \,\a_\tti
$$
has the property that the integer $-2(\a,\a_\tti)/(\a_\tti,\a_\tti) \ge 1$ for every $\a \in \Delta(v)$.  It now follows from the second equality of \eqref{E:dfn_rho} that $\varrho_w(\ttE) > |w|$ if and only if $\a_\tti$ is short.
\end{proof}

\appendix
\section{Dynkin diagrams} \label{S:dynkin}
For the readers convenience we include in Figure \ref{f:dynkin} the Dynkin diagrams of the complex simple Lie algebras.  Recall that: each node corresponds to a simple root $\a_i \in \sS$; two nodes are connected if and only if $\langle \a_i , \a_j \rangle \not=0$ and in this case the number if edges is $|\a_i|^2/|\a_j|^2 \ge 1$ (that is, $i,j$ are ordered so that the inequality holds).  Below, if $G = B_r$, then $r \ge 3$; and if $G = D_r$, then $r \ge 4$.
%
%
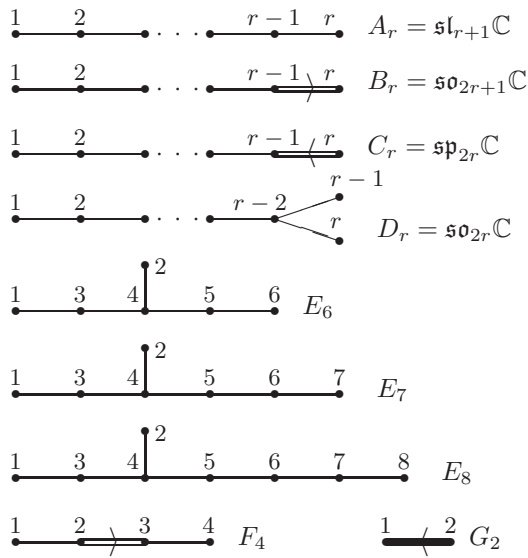
\begin{figure}[!ht] 
\caption{Dynkin diagrams for complex simple Lie algebras.}
\begin{center}
\setlength{\unitlength}{3.5pt}
\begin{picture}(50,65)(0,-5)
\multiput(0,55)(7,0){3}{\circle*{0.8}}
\put(0,55){\line(1,0){14}}
\multiput(15.5,55)(2,0){3}{\circle*{0.3}}
\multiput(21,55)(7,0){3}{\circle*{0.8}}
\put(21,55){\line(1,0){14}}
\put(-0.7,56){\footnotesize{1}}
\put(6.3,56){\footnotesize{2}}
\put(25,56){\footnotesize{$r-1$}}
\put(33.3,56){\footnotesize{$r$}}
\put(38,55){\small{$A_r = \fsl_{r+1}\bC$}}
\multiput(0,49)(7,0){3}{\circle*{0.8}}
\put(0,49){\line(1,0){14}}
\multiput(15.5,49)(2,0){3}{\circle*{0.3}}
\multiput(21,49)(7,0){3}{\circle*{0.8}}
\put(21,49){\line(1,0){7}}
\multiput(28,48.75)(0,0.5){2}{\line(1,0){7}}
\put(31.3,48.4){$\rangle$}
\put(-0.7,50){\footnotesize{1}}
\put(6.3,50){\footnotesize{2}}
\put(25,50){\footnotesize{$r-1$}}
\put(33.3,50){\footnotesize{$r$}}
\put(38,49){\small{$B_r = \fso_{2r+1}\bC$}}
\multiput(0,42)(7,0){3}{\circle*{0.8}}
\put(0,42){\line(1,0){14}}
\multiput(15.5,42)(2,0){3}{\circle*{0.3}}
\multiput(21,42)(7,0){3}{\circle*{0.8}}
\put(21,42){\line(1,0){7}}
\multiput(28,41.75)(0,0.5){2}{\line(1,0){7}}
\put(31.3,41.4){$\langle$}
\put(-0.7,43){\footnotesize{1}}
\put(6.3,43){\footnotesize{2}}
\put(25,43){\footnotesize{$r-1$}}
\put(33.3,43){\footnotesize{$r$}}
\put(38,42){\small{$C_r = \fsp_{2r}\bC$}}
\multiput(0,35)(7,0){3}{\circle*{0.8}}
\put(0,35){\line(1,0){14}}
\multiput(15.5,35)(2,0){3}{\circle*{0.3}}
\multiput(21,35)(7,0){2}{\circle*{0.8}}
\put(21,35){\line(1,0){7}}
\multiput(35,32.67)(0,4.66){2}{\circle*{0.8}}
\put(28,35){\line(3,1){7}}\put(28,35){\line(3,-1){7}}
\put(-0.7,36){\footnotesize{1}}
\put(6.3,36){\footnotesize{2}}
\put(23.5,36){\footnotesize{$r-2$}}
\put(34,38.5){\footnotesize{$r-1$}}
\put(34,34){\footnotesize{$r$}}
\put(39,33){\small{$D_r = \fso_{2r}\bC$}}
\multiput(0,25)(7,0){5}{\circle*{0.8}}
\put(0,25){\line(1,0){28}}
\put(14,30){\circle*{0.8}} 
\put(14,30){\line(0,-1){5}} 
\put(-0.7,26){\footnotesize{1}}
\put(6.3,26){\footnotesize{3}}
\put(15,29){\footnotesize{2}}
\put(12,26){\footnotesize{4}}
\put(20.3,26){\footnotesize{5}}
\put(27.3,26){\footnotesize{6}}
\put(31,24.7){\small{$E_6$}}
\multiput(0,16)(7,0){6}{\circle*{0.8}}
\put(0,16){\line(1,0){35}}
\put(14,21){\circle*{0.8}} \put(14,21){\line(0,-1){5}}
\put(-0.7,17){\footnotesize{1}}
\put(6.3,17){\footnotesize{3}}
\put(15,20){\footnotesize{2}}
\put(12,17){\footnotesize{4}}
\put(20.3,17){\footnotesize{5}}
\put(27.3,17){\footnotesize{6}}
\put(34.3,17){\footnotesize{7}}
\put(39,15.7){\small{$E_7$}}
\multiput(0,7)(7,0){7}{\circle*{0.8}}
\put(0,7){\line(1,0){42}}
\put(14,12){\circle*{0.8}} \put(14,12){\line(0,-1){5}}
\put(-0.7,8){\footnotesize{1}}
\put(6.3,8){\footnotesize{3}}
\put(15,11){\footnotesize{2}}
\put(12,8){\footnotesize{4}}
\put(20.3,8){\footnotesize{5}}
\put(27.3,8){\footnotesize{6}}
\put(34.3,8){\footnotesize{7}}
\put(41.3,8){\footnotesize{8}}
\put(46,6.7){\small{$E_8$}} 
\multiput(0,0)(7,0){4}{\circle*{0.8}} 
\put(0,0){\line(1,0){7}}
\put(7,0.3){\line(1,0){7}} \put(7,-0.3){\line(1,0){7}}
\put(10,-0.7){$\rangle$}
\put(14,0){\line(1,0){7}} \put(24,-0.3){\small{$F_4$}}
\put(-0.7,1){\footnotesize{1}}
\put(6.3,1){\footnotesize{2}}
\put(13.3,1){\footnotesize{3}}
\put(20.2,1){\footnotesize{4}}
\multiput(40,0)(7,0){2}{\circle*{0.8}}
\put(40,0){\line(1,0){7}}\put(40,0.3){\line(1,0){7}}
\put(40,-0.3){\line(1,0){7}}
\put(43.5,-0.7){$\langle$} \put(49,-0.3){\small{$G_2$}}
\put(39.3,1){\footnotesize{1}}
\put(46.3,1){\footnotesize{2}}
\end{picture}
\label{f:dynkin}
\end{center}
\end{figure}
%

\bibliography{../LaTex/refs.bib}	
\bibliographystyle{plain}			

\end{document}

%% file: macros.tex
\def\cf{cf.~}

\newcommand{\hsp}[1]{{\hbox{\hspace{#1}}}}

\newcounter{letcnt} 

\def\a{\alpha}  

\def\d{\delta}

\def\m{\mu}
\def\n{\nu}

\def\w{\omega}

\def\tAut{\mathrm{Aut}}

\def\fb{\mathfrak{b}} 
 
\def\bC{\mathbb C} \def\cC{\mathcal C}

 \def\sD{\mathscr{D}}

\def\td{\mathrm{d}}

 \def\tdim{\mathrm{dim}}
 
 \def\ttE{\mathtt{E}}
\def\tEnd{\mathrm{End}}

\def\texp{\mathrm{exp}}
 
\def\sF{\mathscr{F}}

  \def\ttF{\mathtt{F}}

\def\tGr{\mathrm{Gr}}
\def\fg{{\mathfrak{g}}}

\def\fh{\mathfrak{h}}

\def\tti{\mathtt{i}}

\def\tmax{\mathrm{max}} \def\tmin{\mathrm{min}}

\def\fn{\mathfrak{n}}

 \def\cO{\mathcal O}
 \def\tOG{\mathrm{OG}}

\def\bP{\mathbb P}

\def\fp{\mathfrak{p}}

 \def\cQ{\mathcal Q} 
 
\def\fq{\mathfrak{q}}

\def\sS{\mathscr{S}}
  
\def\ttS{\mathtt{S}} 
\def\fs{\mathfrak{s}}
 
\def\tss{\mathrm{ss}}
\def\tSL{\mathrm{SL}} \def\tSO{\mathrm{SO}}
\def\tSp{\mathrm{Sp}} \def\tSpin{\mathit{Spin}}

 \def\tSym{\mathrm{Sym}}

 \def\tspan{\mathrm{span}}
\def\fsl{\mathfrak{sl}} \def\fso{\mathfrak{so}} 
\def\fsp{\mathfrak{sp}} 
\def\cT{\mathcal T}  
 \def\ttT{\mathtt{T}}

 \def\sW{\mathscr{W}}

   \def\bZ{\mathbb Z}
 
\def\fz{\mathfrak{z}} 
 
\def\half{\tfrac{1}{2}}

\def\tand{\quad\hbox{and}\quad}

\def\dfn{\stackrel{\hbox{\tiny{dfn}}}{=}}
\def\sb{{\hbox{\tiny{$\bullet$}}}}

\def\inj{\hookrightarrow}

\def\op{\oplus}

\newcounter{numcnt}

\newcounter{cnt}

\newcounter{acnt}

\newenvironment{a_list_emph}{ 
  \begin{list}{{\emph{(\alph{acnt})}}}
   {\usecounter{acnt} \setlength{\itemsep}{3pt}
    \setlength{\leftmargin}{25pt} 
    \setlength{\labelwidth}{20pt}
    \setlength{\listparindent}{20pt} }
   }
   {\end{list}}
\newenvironment{a_list_bold}{ 
  \begin{list}{{\bf{(\alph{acnt})}}}
   {\usecounter{acnt} \setlength{\itemsep}{3pt}
    \setlength{\leftmargin}{25pt} 
    \setlength{\labelwidth}{20pt}
    \setlength{\listparindent}{20pt} }
   }
   {\end{list}}
\newcounter{Acnt}

\newcounter{icnt}

\newcounter{Icnt}

\newcounter{exam_cnt}

\newcounter{mccnt}

\newenvironment{bcirclist}{ 
  \begin{list}{\boldmath$\circ$\unboldmath}
   {\usecounter{cnt} \setlength{\itemsep}{2pt}
    \setlength{\leftmargin}{15pt} \setlength{\labelwidth}{20pt}
    \setlength{\listparindent}{20pt} }
   }
   {\end{list}}


%% file: thms.tex
\newtheorem{corollary}[equation]{Corollary}
\newtheorem{lemma}[equation]{Lemma}
\newtheorem{proposition}[equation]{Proposition}
\newtheorem{theorem}[equation]{Theorem}

\theoremstyle{definition}

\newtheorem*{boldQ*}{Question}
\newtheorem*{boldP*}{Problem}

\theoremstyle{definition}

\theoremstyle{remark}
\newtheorem*{assume*}{Assume}
\newtheorem*{answer*}{Answer}

\newtheorem*{claim*}{Claim}

\newtheorem*{definition*}{Definition}
\newtheorem{example}[equation]{Example}
\newtheorem*{example*}{Example}
\newtheorem*{hint*}{Hint}
\newtheorem*{notation*}{Notation}
\newtheorem{remark}[equation]{Remark}
\newtheorem*{remark*}{Remark}
\newtheorem*{remarks*}{Remarks}
\newtheorem*{fact*}{Fact}
\newtheorem*{emphL*}{Lemma}

\newtheorem*{emphQ*}{Question}
\newtheorem*{emphA*}{Answer}
